%% file: main.tex
\documentclass{amsart}[11pt]
\input{packages}
\input{commands}

\title{Strong rate for one-dimensional driftless continuous strict local martingales}

\author{Antoine Jacquier}
\address{Department of Mathematics, Imperial College London}
\email{a.jacquier@imperial.ac.uk}
\date{\today}
\thanks{The authors would like to thank Eduardo Abi Jaber, Jean-Francois Chassagneux, Masaaki Fukasawa and Kostas Kardaras for enlightening discussions and Thibault Jeannin's presentation at the 2026 London-Paris Bachelier workshop where this question arose.}
\keywords{Euler scheme, strict local martingale, strong rate}
\subjclass{65C05,65C30,60H35}

\begin{document}
\maketitle
\begin{abstract}
For one-dimensional
strict local martingales of the form
$\D X_t=\sigma(X_t) \D W_t$ starting from $X_{0}=x_{0}>0$,
we prove the strong rate
$\EE\left[|X^h_N-X_T|\right] \lesssim h^{\min\{\theta\alpha,\half\}}$ 
for a discrete scheme~$X^h$
truncated at a level
$h^{-\alpha}$, with time step~$h$,
as $h\downarrow 0$,
where~$\theta$ depends on the truncation property.
For the CEV model $\sigma(x)\equiv x^{1+\gamma}$ in particular, 
the rate is~$\half$ for any $\gamma \in (0,1)$ (up to a logarithmic pre-factor at $\gamma=\half$).
\end{abstract}
\tableofcontents

\section{Introduction}

In the Fundamental Theorem of Asset Pricing~\cite{delbaen1994constant,delbaen2006general}, the discounted price of a traded asset is required to be a true martingale under the risk-neutral measure. 
When it is not so, in particular when it is a strict local martingale,  major issues arise: Put-Call parity fails, European Call prices lie strictly below their intrinsic values for small strikes 
(leading to undefined implied volatilities) and the pricing PDE may admit several solutions,
as explained in~\cite{cox2005local,ekstrom2009bubbles,heston2007options,jacquier2018implied,jarrow2010asset,kardaras2015strict,pal2010analysis}. 
Such models are standard
descriptions of financial bubbles~\cite{protter2013mathematical},
but may also arise
unintentionally, 
in the case of mis-calibrated (local and/or stochastic) volatility models as in the CEV model~\cite{carr2006jump,cox1996constant,davydov2001pricing},
the SABR model~\cite{hagan2015probability} and other stochastic volatility models with moment explosions~\cite{andersen2007moment,guyon26martingale,sin1998complications}.

For a one-dimensional driftless diffusion $\D X_t = \sigma(X_t)\D W_t$ the strict local martingale
property is equivalent to 
$\int^{\infty} x\sigma(x)^{-2}\D x<\infty$, as devised by Delbaen and Shirakawa~\cite{delbaen2002no}
and Mijatovi\'c and Urusov~\cite{mijatovic2012martingale}, in which case the martingale defect reads 
$x_0-\EE[X_T] = x_0\QQ(\tau_\infty\le T)$ for the (F\"ollmer) measure~$\QQ$ under which the process~$X$ explodes~\cite{carr2014hedging,kardaras2015strict}.
A typical behaviour of such processes is that they `take large values' more often than their true martingale counterparts.
Because of this, 
two distinct obstructions appear when using the usual Euler-Maruyama scheme.
The first is specific to strict local martingales: 
the Euler iterates
$X_{n+1}=X_n+\sigma(X_n)\Delta_n$ form a discrete martingale, 
so $\EE[X_N]=x_0$ for
every step size, whereas $\EE[X_T]<x_0$; hence
$\|X_N-X_T\|_{L^1}\ge x_0-\EE[X_T]>0$ and the scheme cannot converge in~$L^1$. 
The second obstruction is the super-linear growth of~$\sigma$,
so that moments of the Euler iterates diverge in finite
time~\cite{hutzenthaler2011strong}.
The different proposed schemes,
tamed~\cite{hutzenthaler2012strong,sabanis2013note}, 
truncated~\cite{mao2015truncated,mao2016convergence},
implicit~\cite{alfonsi2013strong,neuenkirch2014first},
adaptive
step-sizes~\cite{fang2020adaptive,kelly2018adaptive} ,
work under some coercivity
condition 
or some dominating Lyapunov function arguments. 
However, these conditions
fail for our canonical example $\sigma(x)\equiv x^{1+\gamma}$($\gamma>0$),
as moments greater than~$1$ are infinite. 
Explicit schemes for locally Lipschitz coefficients
without such moment assumptions converge in probability~\cite{hutzenthaler2011convergence} 
but  carry no rate, 
and the lower bounds in~\cite{hefter2019lower} show
that for non-Lipschitz~$\sigma$,
arbitrarily slow strong rates can be unavoidable, 
so that rates below the usual~$\half$ are to be expected.

The closest general framework is the monograph by Hutzenthaler and
Jentzen~\cite{HutzenthalerJentzen15}, 
who observed that the Euler iterates nevertheless satisfy
Lyapunov-type inequalities on the complement of rare events and 
that the probability of leaving this region decays polynomially in the step, which allows to obtain uniform moment bounds for tamed Euler schemes,
but stressing that no rate can be expected at this level of generality. 
Our scheme shares 
this approach by restricting the analysis to high-probability
events and letting the admissible region grow as a negative power of the step, 
which is also the idea behind Mao's truncated scheme~\cite{mao2015truncated,mao2016convergence}.
It differs from~\cite{HutzenthalerJentzen15} though in three aspects:
(i) the admissible Lyapunov
function~$V$ in~\cite{HutzenthalerJentzen15} must satisfy~$V(x)\ge|x|^q$,
that is
$V(x)\equiv 1+x^q$ with $q\le1$ for the CEV model, and the convergence threshold in~\cite[Theorem~3.15]{HutzenthalerJentzen15}  becomes non-positive, 
so no strong convergence holds;
(ii) we do not
tame the increments but kill the scheme at a level~$R_h$ tied to the step;
(iii) we obtain an explicit rate
$\EE[|X^h_N-X_T|]\lesssim h^{\theta\alpha}$, 
with~$\theta$ 
the tail exponent of 
the
truncation error $\Tt(R)=\EE[X_T\ind_{\{\tau_R\le T\}}]$
and~$\gamma$ the growth exponent of~$\sigma$.

There is also a vast literature on killed schemes at a boundary,
for example Gobet~\cite{gobet2000weak,gobet2004exact}
has a weak exact rate~$\half$
and a weak rate~$1$ can be recovered by shifting boundary techniques~\cite{gobet2001euler,gobet2010stopped},
also obtained by 
\c{C}etin~\cite{ccetin2018diffusion} and \c{C}etin and Hok~\cite{ccetin2024speeding} without boundary correction.
All these results concern the weak error
$|\EE[f(X^h_N)]-\EE[f(X_T)]|$ for bounded or smooth~$f$ and fixed domain, whereas our interest is the strong error for the linear payoff with a domain growing
with the time step~$h$. Closer to our purpose, Song and Yang~\cite{song2015approximating} approximate
$\EE[f(X_T)]$ for a strict local martingale by killing~$X$ at a level $R\uparrow\infty$,
and show that the killed functional converges with an explicit rate for a class of functions excluding linear payoff 
(where martingale
defect actually matters).
The link between the
truncation level and the time step, which we heavily exploit here and which produces the final rate, goes back to Mao~\cite{mao2015truncated},
and some of our technical tools are further inspired by Protter, Qiu and San
Mart\'in~\cite{protter2020asymptotic}.

For the continuous driftless one-dimensional strict local martingale
$\D X_t = \sigma(X_t)\D W_t$ starting from $X_0=x_0>0$ 
and an Euler scheme killed upon leaving $(0,R_h)$, $R_h=h^{-\alpha}$,
we prove the strong rate
$$
\EE\left[\left|
X^h_N-X_T\right|\right]
\lesssim h^{\theta\alpha}\ell(h^{-\alpha}) + \sqrt{h},
$$
where $\theta\in(0,1]$ is the decay exponent of the truncation error, $\gamma>0$ the growth
exponent of~$\sigma$ at infinity,
$\ell$ a slowly varying correction,
and we always write~$\lesssim$ to denote inequality up to some multiplication factor independent of~$h$. 
This is the first strong rate for a discretisation scheme of a strict local
martingale; it extends immediately to~$L^p$-norms and to
Lipschitz payoffs such as European Calls,
and we provide an explicit computation of~$\theta$ for the CEV model.
In Section~\ref{sec:setting} we introduce the model, the truncation set-up and the
discretisation scheme and state the main result; Section~\ref{sec:CEV} works out the CEV
model and provides numerical evidence,
while Section~\ref{sec:proofs} gathers the proofs.

 
\section{Setting and main results}\label{sec:setting}
We consider the one-dimensional stochastic differential equation
\begin{equation}\label{eq:sde}
  \D X_t=\sigma(X_t) \D W_t,\qquad X_{0}=x_{0}>0,
\end{equation}
on a filtered probability space $(\Omega,\Ff,(\Ff_t)_{t\ge0},\PP)$ satisfying the usual conditions, 
with~$W$ a one-dimensional $(\Ff_t)$-Brownian motion, and fix a time horizon $T>0$,
and the following assumptions on the diffusion coefficient~$\sigma$:
\begin{assumption}\label{assu:allassu}\ 
\begin{enumerate}[(i)]
\item $\sigma:[0,\infty)\to [0,\infty)$, $\sigma(x)>0$ if $x>0$
and $\int_{a}^{b}\frac{\D x}{\sigma^{2}(x)}<\infty$ for any $0<a<b$;
\item $\displaystyle \sigma\in \Cc^3((0,\infty)\to (0,\infty))$.
\item $\int^{\infty}\frac{x}{\sigma^2(x)} \D x<\infty$.
\item $\int_{0}\frac{x}{\sigma^2(x)} \D x=\infty.$
\end{enumerate}
\end{assumption}

The first condition ensures existence of a unique weak solution~\cite{engelbert1985solutions}, which is a non-negative local martingale, hence a supermartingale.
From~\cite{delbaen2002no, mijatovic2012martingale}, (iii) holds if and only if~$X$ is a
strict local martingale, and we denote
$\df_{T} := x_{0}-\EE[X_T]>0$
the martingale defect;
(iv) is Feller's condition ensuring that the origin cannot be reached with positive probability in finite time.
Finally~(ii) could be relaxed, but simplifies some computations and does not restrict the scope much in regards of applications.
The Lamperti transform
\begin{equation}\label{eq:lamperti}
\Lambda(x):=\int_{x_{0}}^{x}\frac{\D y}{\sigma(y)},\qquad
\qquad\text{with}\qquad
\lf:=\lim_{x\downarrow 0}\Lambda(x)\quad\text{and}\quad \rf:=\lim_{x\uparrow\infty}\Lambda(x),
\end{equation}
is an increasing bijection in $\Cc^4((0,\infty)\to (\lf,\rf))$ and by It\^o's formula, $Z:=\Lambda(X)$
satisfies
\begin{equation}\label{eq:Z}
\D Z_t=\D W_t+b(Z_t) \D t,
\qquad \text{with }
Z_{0} = \Lambda(x_0)
\quad\text{and}\quad
b(z):=-\half  \sigma'\left(\Lambda^{\leftarrow}(z)\right).
\end{equation}

We consider further  assumptions needed to prove a strong rate.
\begin{assumption}\label{assu:osl}
There exists $K_{0}\ge0$ with $\sigma(x)\sigma''(x)\ge-2K_{0}$ for all $x>0$.
Furthermore, 
there exists $x_1\geq x_{0}$ such that
\begin{enumerate}[(i)]
\item on $[x_1,\infty)$:
$\sigma$ is non-decreasing and
there exist $\gamma>0$ and $0<c_{\sigma}\leq \cbar$ such that
$$
 c_{\sigma}  x^{1+\gamma}\leq \sigma(x)\leq \cbar x^{1+\gamma},\qquad
  |\sigma'(x)|\lesssim x^{\gamma},
  \qquad |\sigma''(x)|\lesssim x^{\gamma-1},\qquad
  |\sigma'''(x)|\lesssim x^{\gamma-2};
$$
\item 
on $(0,x_1]$:
\begin{enumerate}[(a)]
\item $\sigma(x)\leq C_{0} x$ for some $C_{0}\ge1$;
\item $|\sigma'|$, $|\sigma\sigma''|$ and $|\sigma^2\sigma'''|$ are bounded;
\item $\sup\{\frac{\sigma(x)}{\sigma(y)}: 0<x\leq x_1, 
\frac{x}{2} \leq y \leq \frac{3x}{2}\}<\infty$.
\end{enumerate}
\end{enumerate}
\end{assumption}

\begin{remark}\ 
\begin{itemize}
\item The first part of the  assumption is equivalent to $b'\leq K_{0}$ on $(\lf,\rf)$.
If~$\sigma$ is convex, it holds with $K_{0}=0$.
\item 
Assumption~\ref{assu:osl}(i) implies Assumption~\ref{assu:allassu}(iii) and $\rf<\infty$.
\item 
Assumption~\ref{assu:osl}(ii)(a) implies Assumption~\ref{assu:allassu}(iv) 
as well as 
\begin{equation}\label{eq:gap0}
\Lambda(x)-\Lambda\left(\frac{x}{2}\right)
= \int_{x/2}^{x}\frac{\D y}{\sigma(y)}
\geq \frac{\log(2)}{C_{0}}\qquad\text{for } x\leq x_1;
\end{equation}
\item 
Assumption~\ref{assu:osl}(ii)(c) is used in Lemma~\ref{lem:measurable}(iv). 
\item 
Assumption~\ref{assu:osl}(ii)(a) implies that $\lf = -\infty$.
\end{itemize}    
\end{remark}

We record the following computations that will come in handy later:
\begin{align}
\Lambda'''\sigma^3 & =2(\sigma')^2-\sigma\sigma''
\qquad\text{and}\qquad
\Lambda''''\sigma^4 = 6\sigma\sigma'\sigma''-6(\sigma')^3-\sigma^2\sigma''',\label{eq:derivatives}\\
b'(z) & = -\half  (\sigma\sigma'')(x),
\quad
b''(z) = -\frac{\sigma\left(\sigma'\sigma''+\sigma\sigma'''\right)(x)}{2},
\quad
(bb')(z) = \frac{\sigma'\sigma\sigma''(x)}{4},
\end{align}
with $x=\Lambda^{\leftarrow}(z)$, 
and
$\{\sigma', 
\sigma\sigma'', 
\Lambda''''\sigma^4, 
b''\circ\Lambda, 
(bb')\circ\Lambda\}\in \Oo\left(1+x^{k\gamma}\right)$
on $(0,\infty)$
with $k\in\{1,2,3,3,3\}$ under Assumption~\ref{assu:osl}.

\subsection{Truncation setup}\label{sec:trunc}
 
For $R>x_{0}$, let $\tau_R:=\inf\{t: X_t = R\}$, 
$X^{(R)}_T:=X_T\ind_{\{\tau_R>T\}}$,
$$
  \Tt(R) := \EE[X_T] - \EE\left[X^{(R)}_T\right] = \EE\left[X_T\ind_{\{\tau_R\leq T\}}\right] \geq 0.
$$
 Under the measure~$\QQ$ defined via $\frac{\D\QQ}{\D\PP}|_{\Ff_t} = \frac{X_t}{x_{0}}$ 
on $\{t<\tau_\infty\}$,
then~\cite{kardaras2015strict}
$$
\D X_t=\frac{\sigma(X_t)^2}{X_t} \D t+\sigma(X_t) \D W^{\QQ}_t,
$$
and
$\QQ(\tau_\infty>T) = \frac{\EE[X_T]}{x_{0}}$ so that $\QQ(\tau_\infty\leq T) = \frac{\df_{T}}{x_{0}}>0$
and
$\Tt(R)=x_{0} \QQ\left(\tau_R\leq T<\tau_\infty\right)$.
The following assumption will be key in determining the strong rate of our scheme:
\begin{assumption}\label{assu:theta}
There exists $\theta\in(0,1]$ 
and a non-decreasing function 
$\ell:(0,\infty) \to [1,\infty)$, slowly varying at infinity,
such that
$\Tt(R)\lesssim R^{-\theta}\ell(R)$ for all $R\ge2x_{0}$.
\end{assumption}

For most cases, considering $\ell(\cdot) \equiv 1$ will be sufficient, the case $\ell \equiv 1\vee\log$ will be needed only for $\CEV(\half)$ in Section~\ref{sec:CEV}.

\begin{remark}\label{rem:thetamom}
Assumption~\ref{assu:theta} holds for some $\theta\in(0,1]$ and~$\ell(\cdot)$ if and only if $\EE[X_T^p]<\infty$ for some $p>1$:
\begin{enumerate}[(i)]
\item For $R>x_{0}$, $\{X_T>R\}\subset\{\tau_R\leq T\}$, thus $R\PP(X_T>R)\leq\Tt(R)$.
Assumption~\ref{assu:theta} implies $\PP(X_T>x)\lesssim x^{-(1+\theta)}\ell(x)$, so that $\EE[X_T^p]<\infty$ for all $p<1+\theta$.
\item If $\EE[X_{T}^{p}]$ is finite for some $p>1$, then
$\Tt(R) = \EE[X_T\ind_{\{\tau_R\leq T\}}]
\leq \|X_{T}\|_{L^p}(\frac{x_{0}}{R})^{1-\frac{1}{p}}$ since $\PP(\tau_{R}\leq T) \leq \frac{x_{0}}{R}$ so Assumption~\ref{assu:theta} holds with $\theta=1-\frac{1}{p}$ and $\ell\equiv 1$.
\end{enumerate}
\end{remark}

\subsection{Discretisation scheme} 
\begin{definition}\label{def:scheme}
Fix $R>x_{0}$, $h=\frac{T}{N}$, $t_n=nh$, 
$\Delta_n := W_{t_{n+1}}-W_{t_n}$, introduce the discrete-time process
$X_{n+1}=X_{n}+\sigma(X_{n})\Delta_n$ 
starting from $X_{0} = x_{0}$, with $\sigma$ extended by $\sigma(y):=0$ for $y\le0$,
and
\begin{equation}\label{eq:nY}
\nbar := \min\{n\ge1: X_{n}\notin(0,R)\},
\qquad X^h_n := X_{n}\ind_{\{n<\nbar\}}.
\end{equation}
\end{definition}

\begin{assumption}\label{assu:LR}
We take $R_h=h^{-\alpha}$ 
for some $\alpha\in(0,\frac{1}{2\gamma})$.
\end{assumption}

\begin{assumption}\label{assu:mom}
There exists $\delta\in(0,1]$ such that
$\EE\left[\sigma(X_T)^{1+\delta}\right]$ is finite.
\end{assumption}
From Assumption~\ref{assu:osl}(i), Assumption~\ref{assu:mom} 
then implies that
$\EE[X_T^{(1+\gamma)(1+\delta)}]$ is also finite, 
hence Assumption~\ref{assu:theta}
holds with $\theta = 1-\frac{1}{(1+\gamma)(1+\delta)}$ by Remark~\ref{rem:thetamom}(ii).
The following result, with proof  postponed to Section~\ref{sec:prelim}, will also be useful:
\begin{lemma}\label{lem:VT}
Let
$V_T :=
\int_0^T\left(1+X_s^{2\gamma}\right)\D s$.
Under Assumptions~\ref{assu:allassu} and~\ref{assu:osl}, there exists $\kappa_{0} > 0$ such that
$\EE[\exp\{\kappa_{0} V_T\}]$ is finite.
\end{lemma}


\subsection{Convergence result}\label{sec:main}
We now state and prove the main convergence result:

\begin{theorem}\label{thm:main}
Under Assumptions~\ref{assu:allassu}-\ref{assu:osl}-\ref{assu:theta}-\ref{assu:LR}-\ref{assu:mom},
for~$h$ small enough,
\begin{equation}\label{eq:main2}
\EE\left[\left|
X^h_N-X_T\right|\right]
\lesssim h^{\theta\alpha}\ell(h^{-\alpha}) + \sqrt{h}.
\end{equation}
Consequently, the strong rate is $\min\{\theta\alpha,\half\}$ for every $\alpha\in(0,\frac{1}{2\gamma})$:
\begin{enumerate}[(i)]
\item if $\theta>\gamma$, the choice $\alpha=\frac{1}{2\theta}$ is admissible and
$\EE[|X^h_N-X_T|]\lesssim \sqrt{h}\ell(h^{-\frac{1}{2\theta}})$;
\item if $\theta\le\gamma$, for every $\eps>0$ the choice $\alpha=\frac{1}{2\gamma}-\eps$ gives
$\EE[|X^h_N-X_T|]\lesssim h^{\frac{\theta}{2\gamma}-\theta\eps}\ell(h^{-\alpha})$.
\end{enumerate}
\end{theorem}
\begin{remark}\label{rem:rem_thm}
In~\eqref{eq:main2}, 
the first term comes from the truncation error while the second originates from the discretisation error.
A third error, corresponding to the probability that the scheme diverges from the true path
before being killed, is bounded by an arbitrarily large power of~$h$
and does not appear.
\end{remark}
The following corollary is straightforward, 
extending the theorem
to an~$L^p$ statement.

\begin{corollary}\label{cor:Lp}
Under the assumptions of Theorem~\ref{thm:main},
\begin{enumerate}[(i)]
\item $\left\|X^h_N-X_T\right\|_{L^p}
\lesssim h^{\theta\alpha}\ell(h^{-\alpha})+\sqrt{h}$ for $p\le1$;
\item for $1<p<\min\{1+\delta, p^{*}\}$ with $p^{*}:=\sup\{p:\EE[X_T^p]<\infty\}$,
$$
\left\|X^h_N-X_T\right\|_{L^p}
\lesssim \EE\left[X_T^p\ind_{\left\{\tau_{\half R_h}\le T\right\}}\right]^{\frac{1}{p}}+\sqrt{h}.
$$
\end{enumerate}
\end{corollary}

It is natural to ask how the strong error~$\mathrm{D}^{s}$ in the theorem extends to the weak approximation $\mathrm{D}^{w}_f:=\left|\EE[f(X_N^h)] - \EE[f(X_T)]\right|$.
If the function~$f$ is Lipschitz
then trivially $\mathrm{D}^{w}_f \leq \textrm{Lip(f)}\mathrm{D}^{s}$, 
which applies to European Call payoffs $f(x) \equiv \max(x-K,0)$ in Mathematical Finance, which was our primary motivation.
For bounded continuous functions, Assumption~\ref{assu:theta} is clearly not needed, and a better rate would (unsurprisingly) follow.
We leave the general case for further analysis, where each term in~\eqref{eq:main2} would reach differently depending on the behaviour of~$f$, as hinted at in Remark~\ref{rem:rem_thm}.

\section{Application: the CEV model}\label{sec:CEV}
We now check how the truncation scheme and the results above apply to a popular example (especially in the mathematical finance literature~\cite{carr2006jump,davydov2001pricing,hagan2015probability}): the constant-elasticity-of-variance (CEV) model~\cite{cox1996constant}.

\subsection{Theoretical checks}

The CEV model corresponds to~\eqref{eq:sde} with
$\sigma(x)=x^{1+\gamma}$,
and we write $X\sim\CEV(\gamma)$.
Since Assumption~\ref{assu:allassu}(iii) clearly holds if and only if $\gamma>0$, 
we shall only be interested in this case for the CEV model.
The following proposition summarises the conditions on~$\gamma$ ensuring  that $\CEV(\gamma)$ satisfies the different assumptions.
When $\gamma>0$, 
then from~\cite{going2003survey,jeanblanc2009mathematical}, $\zeta:=\frac{X^{-\gamma}}{\gamma}$ is a Bessel $\BES(2+\frac{1}{\gamma})$ process and satisfies
$\D \zeta_t =\frac{1+\gamma}{2\gamma\zeta_t} \D t
-\D W_t$.
Consequently,
$\{\zeta = \frac{1}{\gamma} R^{-\gamma}\} = \{X=R\}$ and
\begin{enumerate}[(a)]
\item $\EE[X_T^p]<\infty$ if and only if $p<p^{*}:=2\gamma+1$, 
in which case $\sup_{t\leq T}\EE[X_t^p]<\infty$, 
and $\PP(X_T>x)\sim x^{-(2\gamma+1)}$, so that
$\EE[X_T^{2(1+\gamma)}\ind_{\{X_T\leq \mf\}}]\sim\mf$;
\item Under $\QQ$, $\zeta\sim\BES(2-\frac{1}{\gamma} < 2)$ and hits~$0$ at $\tau_\infty$;
\end{enumerate}

\begin{proposition}\label{prop:CEV}
Consider $\CEV(\gamma)$ with $\gamma>0$. 
Then
\begin{enumerate}[(i)]
\item The first part of Assumption~\ref{assu:osl} with $K_{0}=0$.
\item
Assumption~\ref{assu:osl}(ii) holds.
\item Assumption~\ref{assu:theta} holds with
$\theta=\min\{1, 2\gamma\}$  and $\ell\equiv \ind_{\{\gamma\ne\half\}} 
+ (1\vee\log)\ind_{\{\gamma=\half\}}$.
\item 
Assumption~\ref{assu:mom} holds for every $\delta<\frac{\gamma}{1+\gamma}$, and Theorem~\ref{thm:main} gives
\begin{equation}\label{eq:CEVrates}
\EE\left[\left|X^h_N - X_T\right|\right]
\lesssim
\left\{
\begin{array}{llll}
\sqrt{h}, 
 & \text{if } \gamma\in(0,\half), & \text{with }R_h=h^{-\frac{1}{4\gamma}},\\
 \sqrt{h}|\log(h)|, &  
 \text{if }\gamma=\half, & \text{with }R_h=h^{-\half},\\
 \sqrt{h}, &  \text{if }\gamma\in(\half,1), & \text{with }R_h=h^{-\half},\\
 h^{a} \text{ for all }
 a \in (0,\frac{1}{2\gamma}), &  \text{if }\gamma\ge1, & \text{with }R_h=h^{-a}.
\end{array}
\right.
\end{equation}
\end{enumerate}
\end{proposition}
\begin{remark}
For $\gamma<1$ the scheme achieves the usual Euler strong rate~$\half$ even though 
the model is a strict local martingale with
infinite $p$-moments for $p\ge2\gamma+1$.
For $\gamma\ge1$,
the exact rate $\frac{1}{2\gamma}$ 
remains open here.
\end{remark}

\begin{proof}
(i) is obvious.
For~(ii), take $x_1=1\vee x_{0}$, $C_{0}=x_{1}^{\gamma}$;
$\sigma'(x) = (1+\gamma) x^{\gamma}$, 
$\sigma(x)\sigma''(x) = \gamma(1+\gamma) x^{2\gamma}$, 
$\sigma^2(x)\sigma'''(x) = \gamma(\gamma+1)(\gamma-1)x^{3\gamma}$ 
all vanish at~$0$, 
and $\frac{\sigma(x)}{\sigma(y)} \leq 2^{1+\gamma}$ 
for $y\geq \frac{x}{2}$. 
Under Assumption~\ref{assu:osl},
$\eps_{R} \sim R^{-\gamma}$,
$B_R \lesssim R^{\gamma}$,
and 
$L_R \lesssim R^{2\gamma}$
as $R\uparrow\infty$.

Regarding~(iii),
under~$\QQ$, $\zeta\sim\BES(2-\frac{1}{\gamma})$ with $\zeta_{0}>0$ 
so~\cite{going2003survey}
$\frac{\zeta_{0}^2}{2\tau_{0}^{\zeta}}\sim\Gamma(\nu)$ with $\nu = \frac{1}{2\gamma}$.
By the strong Markov property, $\Delta:=\tau_{0}^{\zeta} - \tau_{\gamma_{R}}^{\zeta}$ 
(with $\gamma_{R}:=\frac{1}{\gamma} R^{-\gamma}$)
is independent of $\Ff_{\tau_{\gamma_{R}}^{\zeta}}$, 
hence of $\tau_{\gamma_{R}}^{\zeta}$, 
and is distributed as the absorption time of~$\zeta$ started from~$\gamma_{R}$.
Therefore
$\frac{\gamma_{R}^2}{2\tau_{0}^{\zeta}}\sim\Gamma(\nu)$ 
under~$\QQ_{\gamma_{R}}$
as well and
$\QQ(\Delta\leq K \gamma_{R}^2)
= \PP(\Gamma(\nu)\geq \frac{1}{2K})$.
For any $\delta>0$,
$$
\left\{T-\delta < \tau_{\gamma_{R}}^{\zeta}\leq T\right\}
\cap \left\{\Delta\leq K \gamma_{R}^2\right\}
\subset
\{T-\delta < \tau_{0}^{\zeta} \leq T+K\gamma_{R}^{2}\}.
$$
Let 
$\mu_{\gamma_{R}}(\delta)
:= \QQ\left(T-\delta < \tau_{\gamma_{R}}^{\zeta}\leq T\right)$.
Since~$\Delta$ and~$\tau_{\gamma_{R}}^{\zeta}$
are independent, 
and choosing~$K$ sufficiently large
to ensure $\PP(\Gamma(\nu)\geq \frac{1}{2K})\geq \half$,
the left-hand side has probability at least $\half\mu_{\gamma_{R}}(\delta)$
and the right-hand side at most 
$c(\delta+K\gamma_{R}^{2})$, 
with~$c$ the supremum of the density of~$\tau_{0}^{\zeta}$ under~$\QQ$, which is finite since $\tau_{0}^{\zeta} = \frac{\zeta_{0}^2}{2G}$ with $G\sim\Gamma(\nu)$,
that is
$\mu_{\gamma_{R}}(\delta)
\leq 2c(\delta + K \gamma_{R}^2)$,
for all $\delta>0$ and all $\gamma_{R}\in (0,\zeta_{0})$.
Now, since $\Delta\sim\frac{\gamma_{R}^2}{2G}$,
with $b:=c\gamma_{R}^2$ and $G\sim\Gamma(\nu)$,
\begin{align*}
\frac{\Tt(R)}{x_{0}}
 = \EE^{\QQ}\left[\mu_{\gamma_{R}}(\Delta)\right]
 & \leq 2b K 
+ \EE^{\QQ}\left[\min\left\{1,\frac{b }{G}\right\}\right]\\
 & \leq 2cK \gamma_{R}^2
+ \PP\left(G<b \right)
+ b \EE\left[\frac{\ind_{\{G\geq b \}}}{G}\right]
 \leq 2b K
+ \begin{cases}
b ^\nu, & \nu<1,\\
-b \log(b ), & \nu=1,\\
b , & \nu>1,
\end{cases}
\end{align*}
since the density of~$G$ reads $\frac{x^{\nu-1}\E^{-x}}{\Gamma(\nu)}$ and $b\in (0,1)$ for~$R$ large enough,
so $\PP(G<b)\leq \frac{b^\nu}{\Gamma(\nu+1)}$;
the expectation term is
$\frac{1}{\Gamma(\nu)}\int_{b}^{\infty}x^{\nu-2}\E^{-x}\D x$ and the three cases correspond to whether~$x^{\nu-2}$ is integrable near the origin.
With $\gamma_{R}\sim R^{-\gamma}$
then $\gamma_{R}^{2\nu}\sim R^{-1}$ and $\gamma_{R}^2\sim R^{-2\gamma}$, 
that is, in light of Assumption~\ref{assu:theta},
$\theta=1$ and $\ell\equiv 1$ if $\gamma>\half$ 
and 
$\theta=1$ and $\ell\equiv 1\vee \log$ if $\gamma=\half$.
Regarding~(iv), $p^{*}=2\gamma+1>1+\gamma$ gives $\EE[\sigma(X_T)^{1+\delta}] = \EE[X_T^{(1+\gamma)(1+\delta)}]<\infty$ if and only if $\delta<\frac{\gamma}{1+\gamma}$.
Then $\theta>\gamma$ when $\gamma<1$, 
in which case 
$\alpha = \frac{1}{2\theta}$ equals $\frac{1}{4\gamma}$ when $\gamma<\half$ and~$\half$ when $\gamma\in[\half,1)$, both smaller than $\frac{1}{2\gamma}$;
for $\gamma \ge 1$, $\theta = 1 \le \gamma$, and~\eqref{eq:CEVrates} follows from Theorem~\ref{thm:main}.
\end{proof}

\begin{remark}\ 
\begin{itemize}
\item As a consistency check, 
taking $\theta=\min\{1, 2\gamma\}$ in
Remark~\ref{rem:thetamom}
gives 
$\EE[X_{T}^{p}]$ finite
for all $p<1+\theta = 1+\min\{1, 2\gamma\} = 2\gamma+1$ for $\gamma<\half$, as in~(a) above.
Since $\theta\le p^{*}-1 = 2\gamma$, the value $\theta=2\gamma$ is optimal in Assumption~\ref{assu:theta} for $\gamma<\half$.
\item From Section~\ref{sec:proofs},
take $\eps_{R} = \frac{1-2^{-\gamma}}{\gamma} R^{-\gamma}$, 
$\Lambda(y) - \Lambda(\frac{y}{2})=y^{-\gamma}\frac{2^{\gamma}-1}{\gamma} \geq \eps_{R}
>\frac{5\lambda}{2}$ for $y\le2R$, and~\eqref{eq:req2} holds.
\item 
Assumption~\ref{assu:mom} holds for every $\delta<\frac{\gamma}{1+\gamma}$ by~(a) before Proposition~\ref{prop:CEV}.
For Lemma~\ref{lem:VT}, 
$V_T = T + \frac{1}{\gamma^{2}}\int_0^T\frac{\D s}{\zeta_s^{2}}$; since $\zeta\sim\BES(2+\frac{1}{\gamma})$ with index~$\frac{1}{2\gamma}$, then
$\EE[\exp\{\kappa\int_0^T\frac{\D s}{\zeta_s^{2}}\}]$ is finite for $\kappa<\frac{1}{8\gamma^2}$ and infinite for $\kappa>\frac{1}{8\gamma^2}$ 
from~\cite[Part~I,\S6]{jeanblanc2009mathematical}.
\end{itemize}
\end{remark}

\subsection{Numerical validation}

We now test the rate in Theorem~\ref{thm:main} 
on $\CEV(\gamma)$
starting from $x_0=1$
and with time horizon $T=1$, for
$\gamma\in\{\frac{1}{4},\half,\frac{3}{4},1,\frac{3}{2}\}$, 
covering the three regimes of
Proposition~\ref{prop:CEV}(iv).
As a reference value, we consider the exact solution,
on the same Brownian path,
from the Bessel
representation in Section~\ref{sec:CEV}:
the transform 
$\zeta := \frac{X^{-\gamma}}{\gamma}$ solves
$\D \zeta_t = \frac{d-1}{2\zeta_t}\D t - \D W_t$ with $d = 2+\frac{1}{\gamma}>2$,
which we discretise on a fine grid of $N^{\star}=2^{14}$ steps with the drift-implicit
Euler scheme
$$
\zeta_{k+1}
 = \half\left\{
 \zeta_k - \Delta^W_{s_k} + \sqrt{(\zeta_k-\Delta^W_{s_k})^2+2(d-1)\delta}\right\},
\quad \text{with }
\Delta^W_{s_k} :=W_{s_{k+1}}-W_{s_k},
\delta = \frac{T}{N^{\star}},
$$
which is positivity preserving and strongly convergent of order one~\cite{alfonsi2013strong,neuenkirch2014first}. 
For the numerical scheme,
we take~$10^5$ paths,
$R_h = h^{-\alpha}$  and
$h=2^{-j}$, $j=2,\dots,10$, with $\alpha=\half$ for $\gamma\in\{\half,\frac34\}$ and, for $\gamma\in\{1,\frac{3}{2}\}$,
$\alpha = \frac{1}{2\gamma}-\eps$ with $\eps=0.05$ and $\eps=\frac{1}{30}$ respectively.
Figure~\ref{fig:cev_rates} shows the error against~$h$ in logarithmic scale, together with the
rate $\min\{\theta\alpha,\half\}$ predicted by~\eqref{eq:CEVrates}, and the numbers are reported in Table~\ref{tab:cev_rates}.
In all cases where $\gamma\neq\half$, 
the observed rate is at least the predicted one, and so is the fit for $\gamma=\half$.
 
\begin{figure}[htb]
\centering
\includegraphics[scale=0.35]{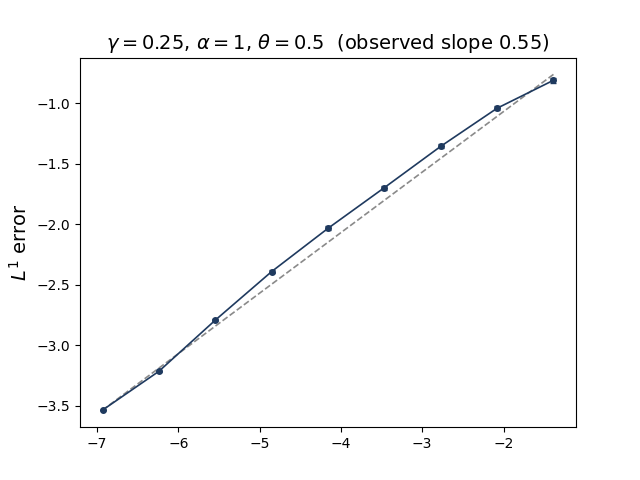}
\includegraphics[scale=0.35]{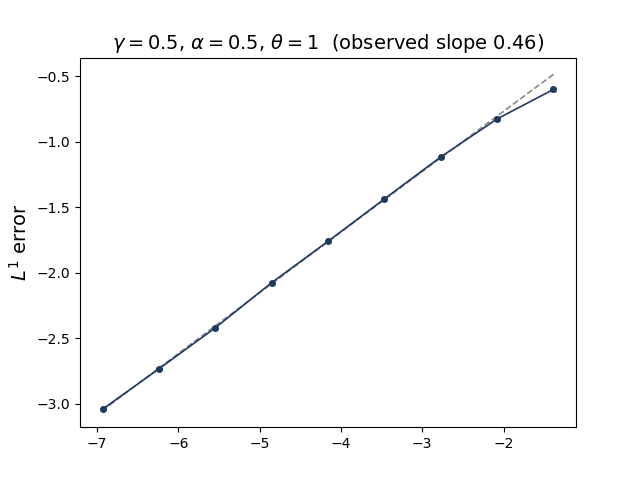}
\includegraphics[scale=0.35]{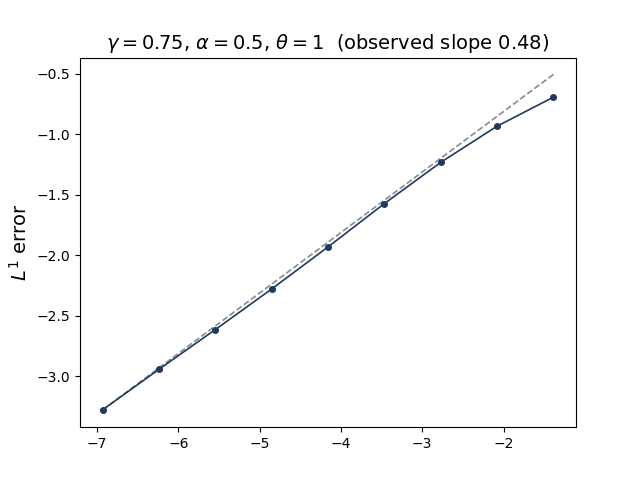}
\includegraphics[scale=0.35]{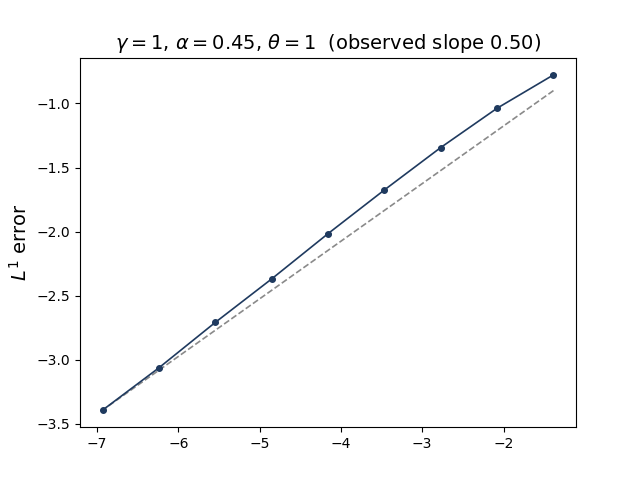}
\includegraphics[scale=0.35]{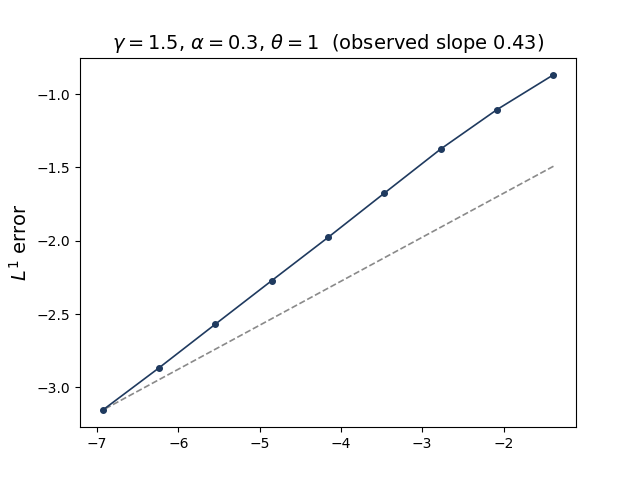}
\caption{Strong error $\mathbb{E}[|X^h_N-X_T|]$  against the
step~$h$ on $\log$ scale 
for different values of~$\gamma$.
For $\gamma=\half$, we fit $\sqrt{h}(a+b\log(1/h))$.
The dashed line has predicted slope~$\min\{\theta\alpha,\half\}$ (Theorem~\ref{thm:main} and
Proposition~\ref{prop:CEV}).}
\label{fig:cev_rates}
\end{figure}
 
\begin{table}[htb]
\centering\small
\begin{tabular}{c c c c c}
\hline
$\gamma$ & $\theta$ & $\alpha$ & predicted slope $\min\{\theta\alpha,\half\}$ & observed slope\\
\hline
$1/4$ & $1/2$ & $1$ & $0.5$ & $0.55$\\
$1/2$ & $1$ & $1/2$ & $\sqrt h\log(1/h)$ & $0.46$\\
$3/4$ & $1$ & $1/2$ & $0.50$ & $0.48$\\
$1$ & $1$ & $0.45$ & $0.45$ & $0.50$\\
$3/2$ & $1$ & $0.30$ & $0.30$ & $0.43$\\
\hline
\end{tabular}
\caption{Predicted and observed $L^1$ rates for $\CEV(\gamma)$; $R_h=2h^{-\alpha}$.}
\label{tab:cev_rates}
\end{table}

For comparison, the plain Euler scheme $X_{n+1} = X_n+\sigma(X_n)\Delta_n$ (with $\sigma(x):=0$ for $x\le0$) was run on the
same paths for $\gamma=1$ with $40,000$ samples. 
The sample mean of~$X_T$ is $0.678$, close to the exact value 
$\EE[X_T] = 2\Phi(1)-1 = 0.683$ since $X\sim 1/\BES(3)$,
so the martingale defect is $\df_T\approx0.32$. 
The~$L^1$ error is around~$2.4$ at $h = 2^{-2}$ and around~$10^{14}$
for $h = 2^{-3}$ as a few paths overshoot to very large values, as explained in~\cite{hutzenthaler2011strong}.

\section{Proofs}\label{sec:proofs}
We introduce a few notations and objects of importance for the proofs.
Define
$$
Z_n := \Lambda(X_{n})\ind_{\{X_{n}>0\}}.
$$
and denote~$I_n$ the interval with endpoints $X_{n}$ and $X_{n+1}$.  
Let $\lambda_{R}:=\Lambda(R)$
and 
$$
\tau^{*} := \tau_{2R} = \inf\left\{t: X_{t} =2R\right\},
\qquad 
\widehat{X} := X^{\tau^{*}}
= X_{\cdot\wedge\tau^{*}},
\qquad \widehat{Z} := \Lambda(\widehat{X}),
$$
so $\widehat{X}$ is a bounded martingale. 
For some $R>0$ large enough, define
\begin{equation}\label{eq:BR_LR}
\eps_{R} := \Lambda(2R)-\Lambda(R),\quad
B_R := \sup_{x\leq 2R}|\sigma'(x)|
\quad\text{and}\quad
L_R := \half \sup_{x\leq 2R}\left|\sigma(x)\sigma''(x)\right|,
\end{equation}
note that Assumption~\ref{assu:osl}(ii)(b) ensures that~$B_R$ and~$L_R$ are both finite, and fix
\begin{equation}\label{eq:clambda}
\lambda := \eps_{R}\min\left\{\frac{1}{4},\frac{2^{-\gamma}}{5}\frac{c_{\sigma}}{\cbar}\right\},
\end{equation}
with~$\cbar,c_{\sigma}$ in
Assumption~\ref{assu:osl},
so $\lambda\lesssim\eps_{R}\sim R^{-\gamma}$. 
For $R\ge1\vee2x_1$ and $\eps_{R}$ small, we have
\begin{align}
\lambda \sfr_R & \leq  \half,\qquad \text{with }
\sfr_R := \sup_{0<y\leq R}\frac{\sigma(y)}{y}\leq C_{0} \vee \cbar R^{\gamma},\label{eq:req1}\\
\frac{5\lambda}{2} & \leq \frac{(2R)^{-\gamma}}{2\cbar} \leq \inf_{2x_1\leq y \leq 2R}\left\{\Lambda(y)-\Lambda\left(\frac{y}{2}\right)\right\}.\label{eq:req2}
\end{align}
Set
$\Df_{n}^{z}:=Z_{n} - \widehat{Z}_{t_n}$,
$\Aa_{n} := \{\sup_{s\in[t_n,t_{n+1}]}|W_s-W_{t_n}|\leq \lambda\}$,
with $\PP(\Aa_{n}^c)\leq \pf_{0} := 4\E^{-\frac{\lambda^2}{2h}}$ 
by Gaussian arguments and the reflection principle, and
\begin{equation}\label{eq:nu}
\nunderbar := \min\Big\{n\ge0: X_{n}\notin(0,R) \text{ or } |\Df_{n}^{z}|>\lambda\text{ or } 
\tau^{*}\leq t_n \text{ or }  \Aa_{n} \text{ fails}\Big\}\wedge N.
\end{equation} 
Note that $\nunderbar\leq \nbar$ 
(defined in~\eqref{eq:nY}), 
and that $|\Df_{n}^{z}|\leq \lambda$ for all $n<\nunderbar$, so 
\begin{equation}\label{eq:Gg_c}
\Gg:=\left\{\max_{n\leq N}|\Df^{z}_{n\wedge\nunderbar}|\leq \lambda\right\}\qquad\text{satisfies}\qquad \Gg^c=\{|\Df^{z}_{\nunderbar}|>\lambda\}.
\end{equation}
Note that $X_{\nunderbar}>0$, 
so~$\Df^{z}_{\nunderbar}$ is well defined under the hypotheses of Lemma~\ref{lem:measurable}: 
if $\nunderbar\geq1$ is triggered by $X_{\nunderbar}\notin(0,R)$ then Lemma~\ref{lem:measurable}(iv) gives $X_{\nunderbar}\geq\half X_{\nunderbar-1}>0$, else $X_{\nunderbar}\in(0,R)$.

\subsection{Preliminaries}\label{sec:prelim}
We start with the following results (Lemmas~\ref{lem:measurable},
\ref{lem:onestep} and~\ref{lem:bound}, and Proposition~\ref{prop:L2})
strongly inspired by~\cite{protter2020asymptotic}.
Introduce 
$$
\Bb_{n} := \{\nunderbar\geq n\}\cap\Big\{X_{n}\in(0,R),|\Df^{z}_{n}|\leq \lambda,\tau^{*}>t_{n}\Big\},
$$
so that
$\{\nunderbar>n\}=\Bb_{n}\cap \Aa_{n}$
with $\Aa_{n}\perp\Ff_{t_n}$
and $\Bb_{n}\in\Ff_{t_n}$.
Indeed, 
$\{\nunderbar>n\}=\bigcap_{k\leq n}\{X_k\in(0,R), |\Df^{z}_k|\leq \lambda, \tau^{*}>t_k\}\cap\bigcap_{k<n}\Aa_k\cap\Aa_n$. 
The first two intersections belong to~$\Ff_{t_n}$,
which is~$\Bb_{n}$. 
The event $\Aa_n$ depends only on $(W_s-W_{t_n})_{s\in[t_n,t_{n+1}]}$, hence is independent of~$\Ff_{t_n}$.

\begin{lemma}\label{lem:measurable}
Let Assumption~\ref{assu:osl} hold, $R\geq 1\vee 2x_1$, and assume~$hB_{R}\leq \lambda$ and
\eqref{eq:req1}-\eqref{eq:req2}.
For $n<N$, on~$\Bb_{n}$,
\begin{enumerate}
\item[(i)]
$X_{n}\in(0,R)$ and $|\Df_{n}^{z}|\leq \lambda$, 
hence 
$\widehat{Z}_{t_n}=Z_{t_n} < \lambda_{R}+\lambda$, 
$\Df_{n}^{z} = \Df^{z}_{n\wedge\nunderbar}$
and 
$X_{n}\le2(X_{t_n}\vee x_1)$ and $X_{t_n}\leq 2(X_{n}\vee x_1)$;
\end{enumerate}
and on $\{\nunderbar>n\}$:
\begin{enumerate}
\item[(ii)] $\tau^{*}>t_{n+1}$ so $\widehat{X}_s=X_s\in(0,2R)$ if $s\leq t_{n+1}$ and $|Z_s-Z_{t_n}|\leq \frac{3\lambda}{2}$ if $s\in[t_n,t_{n+1}]$;
\item[(iii)] $|\Delta_n|\leq \lambda$;
\item[(iv)] $\half X_{n}\leq X_{n+1}\leq \frac{3}{2}X_{n}$, 
so that  $I_n\subset (0,\frac{3R}{2}]$ and $\sigma(X_{n})\leq C \sigma(y)$ for all $y\in[\half X_{n},\frac{3}{2}X_{n}]\supset I_n$, with $C$ depending only on $\sigma$;
\item[(v)] $\sup_{s\in[t_n,t_{n+1}]}X_s\le4(X_{n}\vee x_1)$.
\end{enumerate}
\end{lemma}

\begin{proof}
(i) The first two claims are in the definition of~$\Bb_n$, and $\tau^{*}>t_n$ gives $\widehat{Z}_{t_n}=Z_{t_n}$. If $X_{n}>2x_1$ and $X_{n}>2X_{t_n}$, then $\Df_{n}^{z} =
\Lambda(X_{n}) - \Lambda(X_{t_n}) \geq \Lambda(X_{n}) - \Lambda(\frac{X_{n}}{2})
> \lambda$ by~\eqref{eq:req2} (as $2x_1<X_{n}\leq R$), contradicting $|\Df_{n}^{z}| \leq \lambda$; hence $X_{n}\le2(X_{t_n}\vee x_1)$. Symmetrically, if $X_{t_n}>2x_1$ and $X_{t_n}>2X_{n}$, then 
$-\Df_{n}^{z} \geq \Lambda(X_{t_n})-\Lambda(\frac{X_{t_n}}{2})>\lambda$ by~\eqref{eq:req2} 
(as $X_{t_n}<2R$), again a contradiction.

(ii) Let $s\in(t_n,t_{n+1}]$ with $s\leq \tau^{*}$. 
Then $Z_u=Z_{t_n}+(W_u-W_{t_n})+\int_{t_n}^ub(Z_r)\D r$ for $u\leq s$, and since $X_r\in[0,2R]$ for $r\leq s$, $|b(Z_r)|\leq \half B_R$ there. 
On~$\Aa_n$, $|Z_u-Z_{t_n}|\leq \lambda+\half hB_R\leq \frac{3\lambda}{2}$ for $u\leq s$. 
By~(i), $Z_{t_n} < \lambda_{R}+\lambda$, 
so $Z_u < \lambda_{R}+\frac{5\lambda}{2} < \lambda_{R}+\eps_{R} < \Lambda(2R)$. 
Hence $X_u < 2R$ for all $u\leq s$. 
Applying this with $s=\tau^{*}\wedge t_{n+1}$ shows that $\tau^{*}\wedge t_{n+1}$ is not an exit time, that is $\tau^{*}>t_{n+1}$, 
and the bound holds on all of $[t_n,t_{n+1}]$.

(iii) is immediate from the definition of~$\Aa_n$.

(iv) Write $X_{n+1}=X_{n}(1+\frac{\sigma(X_{n})}{X_{n}}\Delta_n)$. 
By~(i)-(iii)-\eqref{eq:req1}, $|\frac{\sigma(X_{n})}{X_{n}}\Delta_n| \leq \sfr_R\lambda \leq \half$, hence $\half X_{n}\leq X_{n+1}\leq \frac{3}{2}X_{n}$ and $I_n\subset[0,\frac{3R}{2}]$. 
For the statement on~$\sigma$: if $X_{n}\leq x_1$ use Assumption~\ref{assu:osl}(ii)(c); 
if $X_{n}\ge2x_1$, every $y\in [\half X_n, \frac{3}{2}X_n]$ satisfies $y\geq \frac{X_{n}}{2}\geq x_1$, so $\sigma(y)
\geq c_{\sigma} y^{1+\gamma}
\geq c_{\sigma}2^{-(1+\gamma)}X_{n}^{1+\gamma}
\geq \frac{c_{\sigma}}{\cbar}2^{-(1+\gamma)}\sigma(X_{n})$; 
if $x_1<X_{n}<2x_1$, 
then $[\half X_n, \frac{3}{2}X_n]\subset[\frac{x_1}{2},3x_1]$, 
where~$\sigma$ is continuous and positive. 

(v) By~(ii), 
$|\Lambda(X_s)-\Lambda(X_{t_n})|\leq \frac{3\lambda}{2}$ and $X_s<2R$ for $s\in[t_n,t_{n+1}]$; 
if $X_s>2(X_{t_n}\vee x_1)$ then 
$\Lambda(X_s) - \Lambda(\frac{X_s}{2}) > \frac{5\lambda}{2}$ by~\eqref{eq:req2}, a contradiction hence $X_s\le2(X_{t_n}\vee x_1)\le4(X_{n}\vee x_1)$ by~(i).
\end{proof}

On $\{n<\nunderbar\}$,  $\tau^{*}>t_{n+1}$ and $X_{n+1}>0$ by Lemma~\ref{lem:measurable}(ii)-(iv), 
\eqref{eq:Z} gives
\begin{equation}\label{eq:recursion}
\Df^{z}_{n+1} = \Df_{n}^{z} + h\left(b(Z_n)-b(\widehat{Z}_{t_n})\right) + \xi_n,
\end{equation}
with
$\xi_n := b(Z_n)(\Delta_n^2-h) +\Xi_n+\widehat{C}_n$
and
$\widehat{C}_n:=h b(\widehat{Z}_{t_n})-\int_{t_n}^{t_{n+1}}b(\widehat{Z}_s) \D s$.

Note that under Assumptions~\ref{assu:osl}-\ref{assu:LR}, 
as $h\downarrow 0$ then
$hR^{2\gamma} = h^{1-2\gamma\alpha}\downarrow 0$ and $hB_R\lesssim hR^{\gamma} = \sqrt h(hR^{2\gamma})^{\half}\le\lambda\sim R^{-\gamma}$, $hR^{2\gamma}\le1$, $h\le1$ and \eqref{eq:req1}-\eqref{eq:req2} hold for~$h$ small enough. 

\begin{lemma}\label{lem:onestep}
Under Assumptions~\ref{assu:osl} and~\ref{assu:LR}, for~$h$ small enough
on~$\Bb_{n}$,
so that the hypotheses of Lemma~\ref{lem:measurable} hold,
and, with~$C$ depending only on~$\sigma$,
\begin{align}
\left|\EE\left[\xi_n\ind_{\Aa_{n}}\mid\Ff_{t_n}\right]\right| & \lesssim h^2 \left(1+X_{n}^{3\gamma}\right) + h\left(B_R+\sqrt{h}R^{2\gamma}\right)\sqrt{\pf_{0}},\label{eq:xi_mean}\\
\EE\left[|\xi_n|^m\ind_{\Aa_{n}}\mid\Ff_{t_n}\right]
& \leq m!\Big(Ch\left(1+X_n^{\gamma}\right)\Big)^m,
\qquad\text{for every integer } m\ge1.\label{eq:xi_moment}
\end{align}
\end{lemma}

 \begin{proof}
Conditional expectations are taken with respect to $\Bb_n\subset\Ff_{t_n}$, 
where $X_{n}\in (0,R)$ 
and $\tau^{*}>t_n$.
On $\Bb_n\cap\Aa_n=\{\nunderbar>n\}$, Lemma~\ref{lem:measurable}(ii)-(v) apply. 
We first prove~\eqref{eq:xi_mean}; the only place where Assumption~\ref{assu:LR} enters is through the hypothesis $hB_R\le\lambda$ of Lemma~\ref{lem:measurable}.
First, 
$\EE[(\Delta_{n}^2-h)\ind_{\Aa_{n}}] = h\PP(\Aa_{n}^c)-\EE[\Delta_{n}^2\ind_{\Aa_{n}^c}]$ and $\EE[\Delta_{n}^2\ind_{\Aa_{n}^c}]\le(\EE\Delta_{n}^4)^{\half}\PP(\Aa_{n}^c)^{\half}\leq \sqrt3 h\sqrt{\pf_{0}}$; 
since $|b(Z_n)|\leq \half B_R$ on~$\Bb_n$, 
then $\EE[b(Z_n)(\Delta_{n}^2-h)\ind_{\Aa_{n}}\ind_{\Bb_{n}}]=\Oo(hB_R\sqrt{\pf_{0}})$ 
and 
$\EE[|b(Z_n)(\Delta_{n}^2-h)|^2\ind_{\Aa_{n}}\ind_{\Bb_{n}}] \leq  2h^2b(Z_n)^2$.
Since $X_{n+1}>0$ on~$\Bb_n\cap \Aa_{n}$,
a straightforward Taylor expansion gives
$Z_{n+1} = Z_n+\Delta_n- \half\sigma'(X_{n})\Delta_n^2 + \Xi_n$
with 
$\Xi_n = \frac{1}{6}\Lambda'''(X_{n})\sigma(X_{n})^3\Delta_n^3 + \frac{1}{24}\Lambda''''(\xi_n')\sigma(X_{n})^4\Delta_n^4$ for some $\xi_n'\in I_n$.
Since~$\Aa_{n}$ is invariant under $(W_{s}-W_{t_n})_s\mapsto-(W_{s}-W_{t_n})_s$, 
$\EE[\Delta_{n}^3\ind_{\Aa_{n}}]=0$ 
and the cubic term has zero conditional mean. 
By Lemma~\ref{lem:measurable}(iv), $\sigma(X_{n})^4|\Lambda''''(\xi_n')|\leq C^4\sup_{J_n}|\Lambda''''\sigma^4|$ with $J_n:=[\half X_{n},\frac{3}{2}X_{n}]$, 
and by~\eqref{eq:derivatives}, 
$\sup_{J_n}|\Lambda''''\sigma^4|\lesssim 1+X_{n}^{3\gamma}$. 
With $\EE[\Delta_{n}^4] = 3h^2$, the quartic term has conditional mean $\Oo(h^2(1+X_{n}^{3\gamma}))$. 
For the second moment, $|\Xi_n|\leq \frac{C^3}{6}\sup_{J_n}|\Lambda'''\sigma^3| |\Delta_{n}|^3$ from the Taylor expansion and Lemma~\ref{lem:measurable}(iv), and $|\Lambda'''\sigma^3|=|2(\sigma')^2-\sigma\sigma''|\le2B_R^2+2L_R\lesssim R^{2\gamma}$ on $J_n\subset(0,2R]$; 
since $\EE[\Delta_{n}^6]=15h^3$, $\EE[\Xi_n^2\ind_{\Aa_{n}}\mid\Ff_{t_n}]\lesssim h^3 R^{4\gamma}$.

Finally, let $\widehat{C}_n := -\int_{t_n}^{t_{n+1}}(b(\widehat{Z}_s)-b(\widehat{Z}_{t_n}))\D s$.
It\^o's formula now yields
$$
b(\widehat{Z}_s)-b(\widehat{Z}_{t_n})
= \int_{t_n}^{s\wedge\tau^{*}}b'(Z_r)\D W_r + \int_{t_n}^{s\wedge\tau^{*}}\left(bb'+\half b''\right)(Z_r)\D r
=: m_{s} + d_{s}.
$$
Clearly, $\EE[m_s^2\mid\Ff_{t_n}]\lesssim h R^{4\gamma}$ and 
$|\EE[m_s\ind_{\Aa_{n}}\mid\Ff_{t_n}]|=|\EE[m_s\ind_{\Aa_{n}^c}\mid\Ff_{t_n}]|\leq \sqrt{h} R^{2\gamma}\sqrt{\pf_{0}}$,
and Lemma~\ref{lem:measurable}(v) yields the claimed bound for~$d_s$ in~\eqref{eq:xi_mean}.

For~\eqref{eq:xi_moment} we bound~$\xi_n$ pathwise on $\Bb_n\cap\Aa_n$, where $|\Delta_n|\le\lambda\lesssim R^{-\gamma}$ and $X_n\le R$.
Write $\overline{W}_n:=\sup_{s\in[t_n,t_{n+1}]}|W_s-W_{t_n}|$.
By~\eqref{eq:derivatives} and Assumption~\ref{assu:osl}, $|b(\Lambda(x))| \lesssim 1+x^{\gamma}$, $|b'(\Lambda(x))| + |\Lambda'''(x)\sigma^3(x)|
\lesssim1+x^{2\gamma}$ and $|\Lambda''''(x)\sigma^4(x)|\lesssim1+x^{3\gamma}$ for all $x>0$. Hence,
\begin{itemize}
\item $|b(Z_n)(\Delta_n^2-h)|\lesssim (1+X_n^{\gamma})(\Delta_n^2+h)$;
\item from the above Taylor expansion, $\Lambda''''(\xi_n')\sigma(X_n)^4 = (\Lambda''''\sigma^4)(\xi_n')(\frac{\sigma(X_n)}{\sigma(\xi_n')})^4$ and Lemma~\ref{lem:measurable}(iv),
$|\Xi_n|\lesssim \left(1+X_n^{2\gamma}\right)\lambda\Delta_n^2 + \left(1+X_n^{3\gamma}\right)\lambda^2\Delta_n^2
\lesssim \left(1+X_n^{\gamma}\right)\Delta_n^2$ 
because $X_n^{2\gamma}\lambda\lesssim X_n^{\gamma}(X_n/R)^{\gamma}\le X_n^{\gamma}$ and $X_n^{3\gamma}\lambda^2\lesssim X_n^{\gamma}$;
\item $\widehat{Z}_s$ and $\widehat{Z}_{t_n}$ both lie in $\Lambda((0,4(X_n\vee x_1)])$ for $s\in[t_n,t_{n+1}]$ by Lemma~\ref{lem:measurable}(ii)-(v), where $|b'|\lesssim 1+X_n^{2\gamma}$, and $|\widehat{Z}_s-\widehat{Z}_{t_n}|\le \overline{W}_n+\half hB_R$ since $|b|\le\half B_R$ on $\Lambda((0,2R])$; 
thus, with $1+X_n^{2\gamma}\le(1+X_n^{\gamma})(1+R^{\gamma})$, 
$hB_R\lesssim hR^{\gamma}$ and $hR^{2\gamma}\le1$,
$$
\left|\widehat{C}_n\right| \lesssim h\left(1+X_n^{2\gamma}\right)
\left(\overline{W}_n+hB_R\right)
\lesssim \left(1+X_n^{\gamma}\right)h\left(1+(1+R^{\gamma})\overline{W}_n\right).
$$
\end{itemize}
Therefore, on $\Bb_n\cap\Aa_n$, $|\xi_n|\le C(1+X_n^{\gamma})\big(\Delta_n^2+h+h(1+R^{\gamma})\overline{W}_n\big)$.
Conditionally on~$\Ff_{t_n}$, $\frac{\Delta_n^2}{h}$ is the square of a Gaussian, with $m$-th moment $(2m-1)!!\le2^mm!$, and $\frac{\overline{W}_n}{\sqrt{h}}$ has the same law as $\sup_{[0,1]}|W|$, whose $m$-th moment is at most $C^m\sqrt{m!}$.
Since $(1+R^{\gamma})\sqrt{h}\lesssim 1+(hR^{2\gamma})^{\half}\lesssim1$,
we thus deduce~\eqref{eq:xi_moment} from
$$
\EE\left[\left(\Delta_n^2+h+h(1+R^{\gamma})\overline{W}_n\right)^m\Big|\Ff_{t_n}\right]
\le 3^{m-1}h^m\left(2^mm!+1+C^m\sqrt{m!}\right)\le m!(Ch)^m.
$$
\end{proof}

The following lemma ensures that the stopped Lamperti error is bounded, so that all expectations below are finite.

\begin{lemma}\label{lem:bound}
Under the hypotheses of Lemma~\ref{lem:measurable}, $|\Df^{z}_{n\wedge\nunderbar}|\le C\lambda$ for all $n\le N$, with~$C$ depending only on~$\sigma$.
\end{lemma}
\begin{proof}
If $n<\nunderbar$ then $|\Df^{z}_n|\le\lambda$ by~\eqref{eq:nu} and $\Df^{z}_0=0$.
Take $n=\nunderbar\ge1$. 
Either~$\Aa_n$ fails and the other three conditions in~\eqref{eq:nu} hold at~$n$, 
and $|\Df^z_n|\le\lambda$;
or $n=N$ with all conditions satisfied, and $|\Df^z_N|\le\lambda$;
or one event of $\{X_n\notin(0,R), |\Df^z_n|>\lambda, \tau^{*}\le t_n\}$ occurs.
In the last case $n-1<\nunderbar$, so Lemma~\ref{lem:measurable} applies at $n-1$: by~(iv), $X_n\in I_{n-1}\subset[\half X_{n-1},\frac{3}{2}X_{n-1}]$, in particular $X_n>0$, $Z_n=\Lambda(X_n)$ and $\sigma(X_{n-1})\le C\sigma(y)$ for $y\in I_{n-1}$, so
$$
|Z_n-Z_{n-1}| = \left|\int_{X_{n-1}}^{X_n}\frac{\D y}{\sigma(y)}\right|
\le\frac{|X_n-X_{n-1}|}{\inf_{I_{n-1}}\sigma}
= \frac{\sigma(X_{n-1})|\Delta_{n-1}|}{\inf_{I_{n-1}}\sigma}\le C|\Delta_{n-1}|\le C\lambda;
$$
by~(ii), $\tau^{*}>t_n$ and $|\widehat{Z}_{t_n}-\widehat{Z}_{t_{n-1}}| = |Z_{t_n}-Z_{t_{n-1}}|\le\frac{3\lambda}{2}$,
so $|\Df^z_n|\le|\Df^z_{n-1}|+C\lambda+\frac{3\lambda}{2}\le C'\lambda$.
\end{proof}

\begin{proof}[Proof of Lemma~\ref{lem:VT}]
Let $\chi\in\Cc^\infty(\RR)$ non-decreasing, null on $(-\infty,\frac{x_1}{2}]$ 
equal to~$1$ on $[x_1,\infty)$ 
and
$g(x) := \chi(x)\frac{1+x^{2\gamma}}{\sigma(x)^2}$ for $x>0$.
By Assumption~\ref{assu:osl}(i), with $F'(x):=-\int_x^\infty g(y)\D y$ then $F\in\Cc^2((0,\infty))$, $F''=g$, $F'$ is bounded and constant on $(0,\frac{x_1}{2}]$, and $|F'(x)|\lesssim \frac{1}{x}$ for $x\ge x_1$, hence, for all $x>0$,
\begin{equation}\label{eq:Fbounds}
|F(x)-F(x_0)|\le C_F\left(1 + \log(x)_{+}\right)
\qquad\text{and}\qquad
F'(x)^2\sigma(x)^2\le C_M\left(1+x^{2\gamma}\right).
\end{equation}
It\^o's formula gives, for every stopping time $\tau\le T$,
$$
F(X_\tau)-F(x_0) = M_\tau+\half\int_0^\tau\chi(X_s)\left(1+X_s^{2\gamma}\right)\D s,
$$
with
$M_t:=\int_0^t F'(X_s)\sigma(X_s)\D W_s$
and
$\langle M\rangle_t\le C_MV_t$.
Since $\chi\ge\ind_{[x_1,\infty)}$,
\begin{equation}\label{eq:Vrep}
V_\tau\le T\left(1+x_1^{2\gamma}\right)+2\left(F(X_\tau)-F(x_0)\right)-2M_\tau.
\end{equation}
Let $\tau_k:=\inf\{t:X_t\notin(\frac{1}{k},k)\}\wedge T$, so $V_{\tau_k}\le T(1+k^{2\gamma})$; 
since~$X$ does not hit~$0$ nor explodes on $[0,T]$ by Assumption~\ref{assu:allassu}, then $\tau_k\uparrow T$ as $k\uparrow\infty$.
Doob's inequality for the non-negative supermartingale~$X$ 
gives $\PP(\overline{X}>m)\le\frac{x_0}{m}$
with $\overline{X}:=\sup_{t\le T}X_t$, hence $\EE[(1+\overline{X})^r]$ is finite for all $r<1$.
For $\kappa>0$, $\Phi_k(\kappa):=\EE[\exp\{\kappa V_{\tau_k}\}]<\infty$; by~\eqref{eq:Vrep}-\eqref{eq:Fbounds} and Cauchy-Schwarz,
$$
\Phi_k(\kappa)\le C_\kappa\EE\left[\left(1+\overline{X}\right)^{4C_F\kappa}\right]^{\half}\EE\left[\E^{-4\kappa M_{\tau_k}}\right]^{\half},
\quad \text{with }C_\kappa := \exp\left\{\kappa T\left(1+x_1^{2\gamma}\right) + 2\kappa C_F\right\}.
$$
Writing $\E^{-4\kappa M} = \E^{-4\kappa M-16\kappa^2\langle M\rangle}\E^{16\kappa^2\langle M\rangle}$, 
the identity $\EE[\E^{-8\kappa M_{\tau_k}-32\kappa^2\langle M\rangle_{\tau_k}}]\le1$ and Cauchy-Schwarz give
$$
\EE\left[\E^{-4\kappa M_{\tau_k}}\right]\le\EE\left[\E^{32\kappa^2\langle M\rangle_{\tau_k}}\right]^{\half}\le\sqrt{\Phi_k\left(32C_M\kappa^2\right)}.
$$
Pick $\kappa_0>0$ with $4C_F\kappa_0<1$ and $32C_M\kappa_0\le1$. For $\kappa\le\kappa_0$, $32C_M\kappa^2\le\kappa$, so $\Phi_k(32C_M\kappa^2)\le\Phi_k(\kappa)$ hence $\Phi_k(\kappa)\le C'\Phi_k(\kappa)^{\frac{1}{4}}$ with $C':=C_\kappa\EE[(1+\overline{X})^{4C_F\kappa}]^{\half}$ finite,
that is $\Phi_k(\kappa)\le C'^{\frac{4}{3}}$ uniformly in~$k$. 
Monotone convergence as $k\uparrow\infty$ concludes.
\end{proof}

\begin{proposition}\label{prop:L2}
Under Assumptions~\ref{assu:allassu}-\ref{assu:osl}-\ref{assu:LR}, for any $q\ge2$ there is~$C_q$ such that
$$
\EE\left[\max_{n\le N}\left|\Df^z_{n\wedge\nunderbar}\right|^q\right]\le C_qh^{\frac{q}{2}}+\rho_h,
\qquad\text{for~$h$ small enough,}
$$
where~$\rho_h$ denotes a generic quantity of the form $R^aN^b\sqrt{\pf_0}$ with $a,b\ge0$.
\end{proposition}
\begin{remark}
Under Assumption~\ref{assu:LR}, $\rho_{h} = o(h^M)$ for every $M>0$ 
since $\pf_0 = 4\exp\{-\frac{\lambda^2}{2h}\}$ and $\lambda^2\sim R^{-2\gamma} = h^{2\gamma\alpha}$ with $2\gamma\alpha<1$.
\end{remark}

\begin{proof}
Write $m:=\frac{q}{2}\ge1$ and $Y_n:=|\Df^z_{n\wedge\nunderbar}|^2$, so that $0\le Y_n\le C\lambda^2$ by Lemma~\ref{lem:bound};
all the other random variables below are bounded for fixed~$h$ (on $\Bb_k\cap\Aa_k$ the increments are bounded by the estimates in the proof of Lemma~\ref{lem:onestep}, with $|\Delta_k|\le\lambda$ and $S_k\le\lambda$), 
so all expectations are finite. 
Write also $\Df^b_n:=b(Z_n)-b(\widehat{Z}_{t_n})$.

\emph{Step 1.}
On $\{n<\nunderbar\}=\Bb_n\cap\Aa_n$, by~\eqref{eq:recursion} and $\Df^z_n\Df^b_n\le K_0|\Df^z_n|^2$ from Assumption~\ref{assu:osl},
\begin{align*}
|\Df^z_{n+1}|^2 & = |\Df^z_n|^2+2h\Df^z_n\Df^b_n+2\Df^z_n\xi_n+(h\Df^b_n+\xi_n)^2\\
& \le (1+2K_0h)|\Df^z_n|^2+2h^2(\Df^b_n)^2+2\Df^z_n\xi_n+2\xi_n^2.
\end{align*}
By the mean value theorem, $|\Df^b_n|\le\sup|b'||\Df^z_n|$, 
the supremum being over the interval between~$Z_n$ and $\widehat{Z}_{t_n}=\Lambda(X_{t_n})$; 
by Lemma~\ref{lem:measurable}(i), $X_{t_n}\le2(X_n\vee x_1)$, so $\sup|b'|\lesssim1+X_n^{2\gamma}$ by~\eqref{eq:derivatives}. Since moreover $|\Df^z_n|\le\lambda\lesssim R^{-\gamma}$ on~$\Bb_n$, $X_n\le R$ and $R\ge1$, 
then
$$
h^2(\Df^b_n)^2\lesssim h^2\left(1+X_n^{4\gamma}\right)R^{-2\gamma}\lesssim h^2\left(1+X_n^{2\gamma}\right).
$$
Let $\mu_n:=\EE[\xi_n\ind_{\Bb_n\cap\Aa_n}\mid\Ff_{t_n}] = \ind_{\Bb_n}\EE[\xi_n\ind_{\Aa_n}\mid\Ff_{t_n}]$ and $\xi^0_n:=\xi_n\ind_{\Bb_n\cap\Aa_n}-\mu_n$, 
a bounded $(\Ff_{t_{n+1}})_n$-martingale difference  
($\xi_n$ is $\Ff_{t_{n+1}}$-measurable). As $\Bb_n\in\Ff_{t_n}$,
$\ind_{\{n<\nunderbar\}}\Df^z_n\xi_n = \ind_{\Bb_n}\Df^z_n\xi^0_n+\ind_{\Bb_n}\Df^z_n\mu_n$
and $|\Df^z_n|\le\lambda$, $\lambda X_n^{3\gamma}\lesssim R^{-\gamma}R^{\gamma}X_n^{2\gamma}$, $\lambda(B_R+\sqrt hR^{2\gamma})\lesssim R^{\gamma}$ by~\eqref{eq:xi_mean}, 
so that
$$
\ind_{\Bb_n}|\Df^z_n\mu_n|\lesssim\ind_{\Bb_n}\left(h^2\left(1+X_n^{2\gamma}\right)+hR^{\gamma}\sqrt{\pf_0}\right).
$$
Altogether, since $Y_{n+1}-Y_n = (|\Df^z_{n+1}|^2-|\Df^z_n|^2)\ind_{\{n<\nunderbar\}}$,
\begin{equation}\label{eq:onestepY}
Y_{n+1}\le(1+2K_0h)Y_n+2\ind_{\Bb_n}\Df^z_n\xi^0_n+2\xi_n^2\ind_{\Bb_n\cap\Aa_n}+Ch^2\left(1+X_n^{2\gamma}\right)\ind_{\Bb_n}+ChR^{\gamma}\sqrt{\pf_0}.
\end{equation}

\emph{Step 2.}
With $w_n:=(1+2K_0h)^{-n}\in[\E^{-2K_0T},1]$ for $n\le N$, \eqref{eq:onestepY} reads $w_{n+1}Y_{n+1}\le w_nY_n+w_{n+1}(\cdots)$; summing and using $Y_0=0$,
\begin{gather*}
\max_{n\le N}Y_n\le\E^{2K_0T}\left(\max_{n\le N}|M_n|+S_N\right),
\qquad
M_n := 2\sum_{k<n}w_{k+1}\ind_{\Bb_k}\Df^z_k\xi^0_k,\\
S_N:=\sum_{k<N}\left(2\xi_k^2\ind_{\Bb_k\cap\Aa_k}+Ch^2\left(1+X_k^{2\gamma}\right)\ind_{\Bb_k}+ChR^{\gamma}\sqrt{\pf_0}\right),
\end{gather*}
where~$M$ is a martingale.
Hence $\EE[\max_nY_n^m]\lesssim\EE[\max_n|M_n|^m]+\EE[S_N^m]$.

\emph{Step 3.}
Let $Q_N:=\sum_{k<N}\xi_k^2\ind_{\Bb_k\cap\Aa_k}$ and $Q'_N:=\sum_{k<N}\ind_{\Bb_k}(\xi^0_k)^2\le2Q_N+2\sum_{k<N}\ind_{\Bb_k}\mu_k^2$.
By the Burkholder-Davis-Gundy inequality and $\ind_{\Bb_k}|\Df^z_k|^2\le Y_k$ (recall that $k\le\nunderbar$ on~$\Bb_k$),
\begin{align*}
\EE\left[\max_n|M_n|^m\right]
\lesssim \EE\left[\left|\sum_{k<N}\ind_{\Bb_k}|\Df^z_k|^2(\xi^0_k)^2\right|^{\frac{m}{2}}\right]
& \le C \EE\left[\max_nY_n^{\frac{m}{2}} Q_N'^{\frac{m}{2}}\right]\\
&\le\eps \EE\left[\max_nY_n^m\right]+C_\eps \EE\left[Q_N'^{m}\right],
\end{align*}
for every $\eps>0$.
By~\eqref{eq:xi_mean}, $X_k^{6\gamma}\le R^{4\gamma}X_k^{2\gamma}$ and $hR^{2\gamma}\le1$ and hence
$$
\sum_{k<N}\ind_{\Bb_k}\mu_k^2
\lesssim\sum_{k<N}\ind_{\Bb_k}\left(h^4\left(1+X_k^{6\gamma}\right)+h^2R^{4\gamma}\pf_0\right)
\lesssim h\left(hR^{2\gamma}\right)^2V_h+ThR^{4\gamma}\pf_0\le ChV_h+\rho_h.
$$
Choosing~$\eps$ small, it remains to bound $\EE[Q_N^m]$, $\EE[(hV_h)^m]$ and $\EE[S_N^m]$, where $S_N\le2Q_N+ChV_h+\rho_h$.

\emph{Step 4.}
Let $a_k:=\xi_k^2\ind_{\Bb_k\cap\Aa_k}$ and $\bar a_k:=\EE[a_k\mid\Ff_{t_k}] = \ind_{\Bb_k}\EE[\xi_k^2\ind_{\Aa_k}\mid\Ff_{t_k}]$, so that $\bar a_k\lesssim h^2(1+X_k^{2\gamma})\ind_{\Bb_k}$ by~\eqref{eq:xi_moment} with $m=2$, and $\sum_k\bar a_k\lesssim hV_h$.
Then $Q_N=\sum_k\bar a_k+L_N$ with~$L$ a martingale, and by Burkholder-Davis-Gundy and $(a_k-\bar a_k)^2\le2a_k^2+2\bar a_k^2$,
$$
\EE\left[|L_N|^m\right]
\lesssim \EE\left[\left(\sum_ka_k^2\right)^{\frac{m}{2}}\right] + \EE\left[\left(\sum_k\bar a_k^2\right)^{\frac{m}{2}}\right].
$$
Since $\sum_k\bar a_k^2\le\max_k\bar a_k\sum_k\bar a_k\lesssim h^2(1+R^{2\gamma})hV_h\lesssim h^2V_h$, the last term is at most $Ch^m\EE[V_h^{\frac{m}{2}}]$.
For the other one, $\sum_ka_k^2\le(\max_ka_k)Q_N$, so that Young's inequality implies
\begin{align*}
\EE\left[\Big(\sum_ka_k^2\Big)^{\frac{m}{2}}\right] & \le\eps \EE[Q_N^m]+C_\eps \EE\left[\max_ka_k^m\right],\\
\EE\left[\max_ka_k^m\right] & \le\sum_{k<N}\EE\left[\ind_{\Bb_k}\EE\left[\xi_k^{2m}\ind_{\Aa_k}\mid\Ff_{t_k}\right]\right]
\lesssim h^{2m} \EE\left[\sum_{k<N}\ind_{\Bb_k}\left(1+X_k^{2\gamma m}\right)\right],
\end{align*}
using~\eqref{eq:xi_moment} with $2m$ in place of~$m$. As $X_k^{2\gamma m}\le R^{2\gamma(m-1)}X_k^{2\gamma}$, the last expectation is at most $(1+R^{2\gamma(m-1)})\frac{\EE[V_h]}{h}$ and $h^{2m-1}R^{2\gamma(m-1)} \le h^m$. 
Therefore
$\EE[Q_N^m]\lesssim h^m\EE[V_h^m]+\eps\EE[Q_N^m]+h^m\EE[V_h^{\frac{m}{2}}]+h^m\EE[V_h]$,
hence $\EE[Q_N^m]\lesssim h^m(1+\EE[V_h^m])$
for small enough~$\eps$.

\emph{Conclusion.}
By Lemmas~\ref{lem:riemann} and~\ref{lem:VT}, $\EE[V_h^m]\lesssim\EE[V_T^m]+\EE[\Nn^m]\le C+\rho_h$. Steps~2-4 give $\EE[\max_nY_n^m]\lesssim h^m+\rho_h$, which is the claim for even integers~$q=2m$ since $Y_n^m = |\Df^z_{n\wedge\nunderbar}|^q$.
For $q\ge2$ real, 
let $q':=2\lceil q/2\rceil$; Jensen's inequality gives $\EE[\max_n|\Df^z_{n\wedge\nunderbar}|^q]\le\EE[\max_n|\Df^z_{n\wedge\nunderbar}|^{q'}]^{\frac{q}{q'}}\lesssim h^{\frac{q}{2}}+\rho_h^{\frac{q}{q'}}$, and $\rho_h^{q/q'}$ is again of the same form.
\end{proof}

The following is a short technical lemma needed in the conclusion of the previous proof:
\begin{lemma}\label{lem:riemann}
Under the hypotheses of Lemma~\ref{lem:measurable} and $hR^{2\gamma}\le1$, let
$\Nn:= \sum_{k<N}\ind_{\Aa_k^c}$ and $V_h:=h\sum_{k<N}(1+X_k^{2\gamma})\ind_{\Bb_k}$.
Then $V_h\le C(V_T+\Nn)$ with~$V_T$ as in Lemma~\ref{lem:VT}, and $\EE[\Nn^m]\le N^m\pf_0$ for every $m\ge1$.
\end{lemma}
\begin{proof}
On $\Bb_k\cap\Aa_k$,  $X_k\le2(X_{t_k}\vee x_1)$
by Lemma~\ref{lem:measurable}(i) and Lemma~\ref{lem:measurable}(ii) gives $\tau^{*}>t_{k+1}$ and $|Z_s-Z_{t_k}|\le\frac{3\lambda}{2}$ for $s\in[t_k,t_{k+1}]$.
If $X_{t_k}\ge2x_1$ and $X_s<\half X_{t_k}$ for some such~$s$, then $Z_{t_k}-Z_s>\Lambda(X_{t_k})-\Lambda(\half X_{t_k}) \ge \frac{5\lambda}{2}$ by~\eqref{eq:req2} (since $X_{t_k}<2R$), 
which is a contradiction; 
hence $X_s\ge\half X_{t_k}$ on $[t_k,t_{k+1}]$ and $hX_{t_k}^{2\gamma}\le2^{2\gamma}\int_{t_k}^{t_{k+1}}X_s^{2\gamma}\D s$.
If $X_{t_k}<2x_1$, $hX_{t_k}^{2\gamma}\lesssim h$. In both cases
$h(1+X_k^{2\gamma})\ind_{\Bb_k\cap\Aa_k}\le C\int_{t_k}^{t_{k+1}}(1+X_s^{2\gamma})\D s$.
On $\Bb_k\cap\Aa_k^c$, $X_k<R$ and $h(1+X_k^{2\gamma})\le h+hR^{2\gamma}\le2$ for $h\le1$.
Summing over~$k$ gives the first claim and the second follows since $\Nn\le N$ and $\EE[\Nn]\le N\pf_0$.
\end{proof}

\subsection{Proof of Theorem~\ref{thm:main}}
The proof of the theorem hinges upon the following lemma, which strongly relies on tools developed in~\cite{song2015approximating} and~\cite{ccetin2024speeding}.

\begin{lemma}\label{lem:sync}
Under the hypotheses of Lemma~\ref{lem:measurable}, let $\Aa_{<N} := \bigcap_{n<N}\Aa_{n}$ and
$\pf:=\PP(\Aa_{<N}^{c})$. 
On $\Aa_{<N}\cap\Gg$ (with~$\Gg$ in~\eqref{eq:Gg_c}):
\begin{enumerate}[(i)]
\item if $\tau^{*}\leq T$ then $\nbar\leq N$;
\item if $\nbar\leq N$ and $\tau^{*}>T$ then $X_{\nbar}\geq R$ and~$X$ has reached
$\Lambda^{\leftarrow}(\lambda_{R}-\lambda)$ before~$T$;
\item if $\nbar>N$ and $\tau^{*}>T$ then $\nunderbar=N$.
\end{enumerate}
\end{lemma}

Under Assumptions~\ref{assu:osl} and~\ref{assu:LR}, $\lambda^2\sim R^{-2\gamma} = h^{2\gamma\alpha}$ with $2\gamma\alpha<1$, thus $\pf\leq N\pf_{0} = 4N\exp\{-c h^{-(1-2\gamma\alpha)}\}$ is negligible (it is of the form~$\rho_h$ of Proposition~\ref{prop:L2}).
 
\begin{proof}
Work on $\Aa_{<N}\cap\Gg$. 
On~$\Gg$, $|\Df_{n}^{z}|>\lambda$ cannot be the first to occur in the definition of~$\nunderbar$ in~\eqref{eq:nu} and therefore
\begin{equation}\label{eq:nuon_Gg}
\nunderbar = \nbar\wedge\min\{n:\tau^{*}\leq t_n\}\wedge N ,
\end{equation}
which gives~(iii). 

(i) Suppose $\tau^{*}\leq T$ and let $n_{0}<N$ be such that 
$\tau^{*}\in(t_{n_{0}},t_{n_{0}+1}]$,
so that $\min\{n:\tau^{*}\leq t_{n}\} = n_{0}+1\leq N$. 
Thus $\nunderbar = \nbar\wedge (n_{0}+1)$.
If $\nbar>n_{0}$, then $\nunderbar = n_{0}+1$
and Lemma~\ref{lem:measurable}(ii) at $n=n_{0}$ gives $\tau^{*}>t_{n_{0}+1}$, 
which is a contradiction, and thus $\nunderbar\leq n_{0}$.
Since $\nunderbar = \nbar\wedge (n_{0}+1)$, this implies $\nbar\leq n_{0}<N$.

(ii) Suppose $\nbar\leq N$ and $\tau^{*}>T$. By~\eqref{eq:nuon_Gg}, $\nunderbar=\nbar$, and Lemma~\ref{lem:measurable}(iv) at $n=\nbar-1<\nunderbar$ gives $X_{\nbar}\geq \half X_{\nbar-1}>0$, so $X_{\nbar} \geq R$, $Z_{\nbar}$ and $\Df^{z}_{\nbar}$ are defined, with $|\Df^{z}_{\nbar}| \leq \lambda$ on $\Gg$. 
Since $Z_{\nbar} \geq \lambda_{R}$, $\widehat{Z}_{t_{\nbar}} = Z_{t_{\nbar}} \geq \lambda_{R}-\lambda$, that is 
$X_{t_{\nbar}}\geq \Lambda^{\leftarrow}(\lambda_{R}-\lambda)$, which is~(ii).
\end{proof}

\begin{proof}[Proof of Theorem~\ref{thm:main}]
Under Assumptions~\ref{assu:osl}-\ref{assu:LR}, 
$hB_R\lesssim hR^{\gamma} = \sqrt{h}(hR^{2\gamma})^{\half}\le\lambda\sim R^{-\gamma}$ as $R\uparrow\infty$ and \eqref{eq:req1}-\eqref{eq:req2} hold for all~$h$ small enough, so  Lemmas~\ref{lem:measurable}-\ref{lem:sync} and Proposition~\ref{prop:L2} apply and $\pf\le N\pf_0=\rho_h$.
Fix $p:=(1+\gamma)(1+\delta)>1$, so that $\EE[X_T^p]<\infty$ by Assumption~\ref{assu:mom}, and let~$p'$ be its conjugate exponent.
Write $\Df := X^h_N-X_T$.
Since $\tau^{*} = \tau_{2R}$, then
\begin{equation}\label{eq:XtauStarleqT}
\EE\left[X_T\ind_{\{\tau^{*}\leq T\}}\right]
= \Tt(2R) \leq C R^{-\theta}\ell(R),
\end{equation}
by Assumption~\ref{assu:theta}.
Denote $\Ee := \Aa_{<N}\cap\Gg$;
since $|\Df| \leq X_N^h + X_{T}$ and
$X_N^h \in [0,R)$, H\"older's inequality gives
$\EE\left[|\Df|\ind_{\Ee^c}\right]
 \leq R \PP(\Ee^c)
  + \|X_T\|_{L^p}\PP(\Ee^c)^{\frac{1}{p'}}$.
By Markov's inequality and Proposition~\ref{prop:L2}, 
for any $q\ge2$,
\begin{align}\label{eq:Lambda}
\PP(\Ee^c)
\le\PP(\Aa_{<N}^c)+\PP(\Gg^c)
\le N\pf_0+\lambda^{-q}\EE\left[\max_{n\le N}|\Df^z_{n\wedge\nunderbar}|^q\right]
 & \lesssim \left(\frac{h}{\lambda^2}\right)^{\frac{q}{2}} + \rho_h\nonumber\\
 & \lesssim h^{\frac{q(1-2\gamma\alpha)}{2}}+\rho_h,
\end{align}
because $\lambda^2\gtrsim R^{-2\gamma} = h^{2\gamma\alpha}$. Since $1-2\gamma\alpha>0$, choosing~$q$ with $\frac{q(1-2\gamma\alpha)}{2}\ge\half+\alpha$ and $\frac{q(1-2\gamma\alpha)}{2p'}\ge\half$ gives
$\EE[|\Df|\ind_{\Ee^c}]\lesssim\sqrt{h}$.
Write~$\Ee$ as the disjoint union $\Ee = \Ee_1\cup\Ee_2\cup\Ee_3$ with
$$
\Ee_1 := \Ee\cap\{\tau^{*}\leq T\},\qquad
\Ee_2 := \Ee\cap\{\tau^{*}>T, \nbar\leq N\},\qquad
\Ee_3 := \Ee\cap\{\tau^{*}>T,
\nbar>N\}.
$$
On $\Ee_{1}$, 
$\nbar\leq N$ by Lemma~\ref{lem:sync}(i), so $X^h_N=0$, $|\Df|=X_T$. 
By~\eqref{eq:XtauStarleqT},
$\EE[|\Df|\ind_{\Ee_{1}}]\lesssim R^{-\theta}\ell(R)$.

Consider now~$\Ee_{2}$.
Take $R' = \Lambda^{\leftarrow}(\lambda_{R}-\lambda)$ to be the level of Lemma~\ref{lem:sync}(ii): by~\eqref{eq:req2} at $y=R\in[2x_1,2R]$, $\Lambda(R)-\Lambda(\frac{R}{2})\geq \frac{5\lambda}{2}>\lambda$, so $R'>\frac{R}{2}$. 
Hence, by Assumption~\ref{assu:theta}, monotonicity of~$\Tt$ and $\frac{R}{2}\ge2x_{0}$,
$\EE\left[X_T\ind_{\{\tau_{R'}\leq T\}}\right]
= \Tt(R') \leq \Tt\left(\frac{R}{2}\right)\leq CR^{-\theta}\ell(R)$.
On~$\Ee_{2}$, 
again $X^h_N=0$ and $|\Df|=X_T$, and by Lemma~\ref{lem:sync}(ii) $X$ reaches~$R'$ before~$T$, 
that is $\Ee_{2}\subset\{\tau_{R'}\leq T\}$,
and therefore
$\EE[|\Df|\ind_{\Ee_{2}}]\leq CR^{-\theta}\ell(R)$.

On~$\Ee_{3}$, $X^h_N=X_N\in(0,R)$ and $X_T\in(0,2R)$, with $\widehat{Z}_T=Z_T=\Lambda(X_T)$. 
By Lemma~\ref{lem:sync}(iii), $\nunderbar=N$, hence $\Df^{z}_N = \Df^{z}_{N\wedge\nunderbar}$ and, on~$\Gg$, $|\Df^{z}_N|\leq \lambda$.
We first show $X_N\le2(X_T\vee x_1)$. 
Otherwise $X_N>2x_1$ and $X_N>2X_T$; 
since $X_N\in[2x_1,2R]$, \eqref{eq:req2} gives
$\Df^{z}_N=\Lambda(X_N)-\Lambda(X_T)\geq \Lambda(X_N)-\Lambda(\frac{X_{N}}{2})\geq \frac{5\lambda}{2}>\lambda$, contradicting $|\Df^{z}_N|\leq \lambda$.
Now, with $\overline{\sigma}(x) := \sup_{z\leq x}\sigma(z)$, 
Assumption~\ref{assu:osl} implies that 
there exists $C_{2}>0$ such that
\begin{equation}\label{eq:sigmabarleq}
\overline{\sigma}(x)\leq C_1 + \sigma(x)
\qquad\text{and}\qquad
\overline{\sigma}(2x) \leq C_{2}\left(1 + \overline{\sigma}(x)\right),
\qquad\text{for all }x>0,
\end{equation}
with $C_1:=\sup_{(0,x_1]}\sigma\le C_0x_1$.
By the mean value theorem applied to~$\Lambda^{\leftarrow}$
(using $(\Lambda^{\leftarrow})' \equiv \sigma\circ\Lambda^{\leftarrow}$), we then have
\begin{align*}
|\Df| = |\Lambda^{\leftarrow}(Z_N)-\Lambda^{\leftarrow}(Z_T)|
 & \leq \overline{\sigma}(X_N\vee X_T) |\Df^{z}_N|
\leq \overline{\sigma}(2(X_T\vee x_1))|\Df^{z}_N|\\
 & \leq C_2(1+\overline{\sigma}(X_T\vee x_1))|\Df^{z}_N|
\leq C(1+\sigma(X_T))|\Df^{z}_N|,
\end{align*}
since~$\overline{\sigma}$ is non-decreasing, \eqref{eq:sigmabarleq}
 and $\overline{\sigma}\leq C_1 + \sigma$. 
By H\"older's inequality with the exponents $1+\delta$ and $q:=\frac{1+\delta}{\delta}\ge2$, Assumption~\ref{assu:mom} and Proposition~\ref{prop:L2},
$$
\EE\left[|\Df|\ind_{\Ee_{3}}\right]\lesssim  \left\|1+\sigma(X_T)\right\|_{L^{1+\delta}}
\EE\left[\max_{n\le N}\left|\Df^{z}_{n\wedge\nunderbar}\right|^q\right]^{\frac{1}{q}}
\lesssim \sqrt{h}+\rho_h^{\frac{1}{q}}.
$$
Adding the different steps gives~\eqref{eq:main2}.
The two choices of~$\alpha$ follow by comparing $\theta\alpha$ with~$\half$ under the constraint $\alpha<\frac{1}{2\gamma}$: if $\theta>\gamma$ then $\frac{1}{2\theta}<\frac{1}{2\gamma}$ and $\theta\alpha=\half$ at $\alpha=\frac{1}{2\theta}$; otherwise $\theta\alpha<\frac{\theta}{2\gamma}\le\half$ for every admissible~$\alpha$.
\end{proof}

\subsection{Proof of Corollary~\ref{cor:Lp}}
For $1<p<\min\{1+\delta,p^{*}\}$, with the notation from the proof of Theorem~\ref{thm:main},
$\|\Df\|_{L^p}\le\|\Df\ind_{\Ee^c}\|_{L^p}+\|\Df\ind_{\Ee_1\cup\Ee_2}\|_{L^p}+\|\Df\ind_{\Ee_3}\|_{L^p}$.
On~$\Ee^c$, $|\Df|\le R+X_T$, so that for $p<p_1<p^{*}$, by H\"older's inequality and~\eqref{eq:Lambda},
$\|\Df\ind_{\Ee^c}\|_{L^p}\le R \PP(\Ee^c)^{\frac{1}{p}}+\|X_T\|_{L^{p_1}}\PP(\Ee^c)^{\frac{1}{p}-\frac{1}{p_1}}\lesssim\sqrt{h}$ for~$q$ large enough.
On $\Ee_1\cup\Ee_2$, $|\Df|=X_T$ and $\Ee_1\cup\Ee_2\subset\{\tau_{R'}\le T\}\subset\{\tau_{R/2}\le T\}$ (with~$R'$ as in the proof of the theorem), so that
$\|\Df\ind_{\Ee_1\cup\Ee_2}\|_{L^p}\le\EE[X_T^p\ind_{\{\tau_{R/2}\le T\}}]^{\frac{1}{p}}$.
On~$\Ee_3$, $|\Df|\le C(1+\sigma(X_T))|\Df^z_{N\wedge\nunderbar}|$, and H\"older's inequality with the exponents $\frac{1+\delta}{p}$ and $q:=\frac{p(1+\delta)}{1+\delta-p}\ge2$ (which requires $p<1+\delta$) together with Proposition~\ref{prop:L2} give
$\|\Df\ind_{\Ee_3}\|_{L^p}\lesssim\|1+\sigma(X_T)\|_{L^{1+\delta}}\EE[\max_n|\Df^z_{n\wedge\nunderbar}|^q]^{\frac{1}{q}}\lesssim\sqrt{h}$.

\bibliography{references}
\bibliographystyle{siam}

\end{document}

%% file: packages.tex
\usepackage[margin=4cm]{geometry}
\usepackage{amsmath, amsfonts, amsthm, amssymb, mathtools}
\usepackage{mathrsfs}
\usepackage{subfig}
\usepackage{stmaryrd}
\usepackage{verbatim}
\usepackage{hyperref}
\usepackage{xcolor}
\usepackage{graphicx}
\usepackage{algorithmic, algorithm}
\usepackage{enumerate}
\usepackage{listings}
\usepackage[foot]{amsaddr}
\usepackage{braket}

\graphicspath{{Figures/}}

\numberwithin{equation}{section}
\numberwithin{figure}{section}
\numberwithin{table}{section}
\usepackage[textsize=tiny]{todonotes}

\makeatletter
\def\@tocline#1#2#3#4#5#6#7{\relax
  \ifnum #1>\c@tocdepth 
  \else
    \par \addpenalty\@secpenalty\addvspace{#2}%
    \begingroup \hyphenpenalty\@M
    \@ifempty{#4}{%
      \@tempdima\csname r@tocindent\number#1\endcsname\relax
    }{%
      \@tempdima#4\relax
    }%
    \parindent\z@ \leftskip#3\relax \advance\leftskip\@tempdima\relax
    \rightskip\@pnumwidth plus4em \parfillskip-\@pnumwidth
    #5\leavevmode\hskip-\@tempdima
      \ifcase #1
       \or\or \hskip 1em \or \hskip 2em \else \hskip 3em \fi%
      #6\nobreak\relax
    \dotfill\hbox to\@pnumwidth{\@tocpagenum{#7}}\par
    \nobreak
    \endgroup
  \fi}
\makeatother

\hypersetup{colorlinks,linkcolor=blue!55!black,citecolor=blue!55!black,urlcolor=blue!55!black}

%% file: commands.tex
\mathtoolsset{showonlyrefs=true}
\newtheorem{theorem}{Theorem}[section]
\newtheorem{corollary}[theorem]{Corollary}
\newtheorem{lemma}[theorem]{Lemma}
\newtheorem{proposition}[theorem]{Proposition}

\theoremstyle{definition}
\newtheorem{definition}[theorem]{Definition}
\newtheorem{remark}[theorem]{Remark}

\newtheorem{assumption}[theorem]{Assumption}

\newcommand{\ind}{1\hspace{-2.1mm}{1}} 
\newcommand{\half}{\frac{1}{2}} 

\newcommand{\RR}{\mathbb{R}}

\newcommand{\PP}{\mathbb{P}}
\newcommand{\QQ}{\mathbb{Q}}
\newcommand{\EE}{\mathbb{E}}

\newcommand{\Aa}{\mathcal{A}}
\newcommand{\Bb}{\mathcal{B}}
\newcommand{\Cc}{\mathcal{C}}

\newcommand{\Ee}{\mathcal{E}}
\newcommand{\Ff}{\mathcal{F}}
\newcommand{\Gg}{\mathcal{G}}

\newcommand{\Nn}{\mathcal{N}}
\newcommand{\Oo}{\mathcal{O}}

\newcommand{\Tt}{\mathcal{T}}

\newcommand{\D}{\mathrm{d}}
\newcommand{\E}{\mathrm{e}}

\newcommand{\eps}{\varepsilon}

\newcommand{\mf}{\mathfrak{m}}
\newcommand{\sfr}{\mathfrak{s}}

\newcommand{\CEV}{\mathrm{CEV}}
\newcommand{\BES}{\mathrm{BES}}
\newcommand{\df}{\mathfrak{d}}

\newcommand{\pf}{\mathfrak{p}}
\newcommand{\lf}{\mathfrak{l}}
\newcommand{\rf}{\mathfrak{r}}
\newcommand{\Df}{\mathfrak{D}}

\newcommand{\nbar}{\overline{n}}
\newcommand{\nunderbar}{\underline{n}}
\newcommand{\cbar}{\overline{c}_{\sigma}}